\documentclass[12pt,a4paper, leqno]{article}
\usepackage{fullpage}
\usepackage{amsmath}
\usepackage{amssymb}
\usepackage{amsthm}
\usepackage{mathrsfs}
\newtheorem{theorem}{Theorem}[section]

\newtheorem{proposition}[theorem]{Proposition}
\newtheorem{lemma}[theorem]{Lemma}

\newtheorem{remark}[theorem]{Remark}

\numberwithin{equation}{section}

\begin{document}

\author{David Windisch}\title{\normalsize \textbf{Random walk on a discrete torus and random interlacements}}
\author{\normalsize David Windisch}
\date{}
\maketitle

\thispagestyle{empty}

\begin{abstract}
We investigate the relation between the local picture left by the trajectory of a simple random walk on the torus $({\mathbb Z} / N{\mathbb Z})^d$, $d \geq 3$, until $uN^d$ time steps, $u > 0$, and the model of random interlacements recently introduced by Sznitman \cite{int}. In particular, we show that for large $N$, the joint distribution of the local pictures in the neighborhoods of finitely many distant points left by the walk up to time $uN^d$
converges to independent copies of the random interlacement at level $u$.
\end{abstract}

\vspace{60 mm}

\noindent Departement Mathematik \hspace{85 mm} February 2008\\
ETH Z\"urich \\
CH-8092 Z\"urich \\
Switzerland \\

\newpage

\setcounter{page}{1}

\section{Introduction}

The object of a recent article by Benjamini and Sznitman \cite{BS07} was to investigate the vacant set left by a simple random walk on the $d \geq 3$-dimensional discrete torus of large side-length $N$ up to times of order $N^d$. The aim of the present work is to study the connections between the microscopic structure of this set and the model of random interlacements introduced by Sznitman in \cite{int}. Similar questions have also recently been considered in the context of random walk on a discrete cylinder with a large base, see \cite{S08}. 

In the terminology of \cite{int}, the interlacement at level $u \geq 0$ is the trace left on ${\mathbb Z}^d$ by a cloud of paths constituting a Poisson point process on the space of doubly infinite trajectories modulo time-shift, tending to infinity at positive and negative infinite times. The parameter $u$ is a multiplicative factor of the intensity measure of this point process. The interlacement at level $u$ is an infinite connected random subset of ${\mathbb Z}^d$, ergodic under translation. Its complement is the so-called vacant set at level $u$. In this work, we consider the distribution of the local pictures of the trajectory of the random walk on $({\mathbb Z}/N{\mathbb Z})^d$ running up to time $uN^d$ in the neighborhood of finitely many points with diverging mutual distance as $N$ tends to infinity. We show that the distribution of these sets converges to the distribution of independent random interlacements at level $u$. 

\vspace{5mm}
In order to give the precise statement, we introduce some notation. For $N \geq 1$, we consider the integer torus
\begin{align}
{\mathbb T} = ({\mathbb Z}/N{\mathbb Z})^d, \quad d \geq 3. \label{def:torus}
\end{align}
We denote with $P_x$, $x \in {\mathbb T}$, resp.~$P$, the canonical law on ${\mathbb T}^{\mathbb N}$ of simple random walk on $\mathbb T$ starting at $x$, resp.~starting with the uniform distribution $\nu$ on $\mathbb T$. The corresponding expectations are denoted by $E_x$ and $E$, the canonical process by $X_.$. Given $x \in {\mathbb T}$, the vacant configuration left by the walk in the neighborhood of $x$ at time $t \geq 0$ is the $\{0,1\}^{{\mathbb Z}^d}$-valued random variable
\begin{align}
\omega_{x,t}(.) = 1 \{ X_m \neq \pi_{\mathbb T}(.) + x, \textrm{ for all } 0 \leq m \leq [t] \}, \label{def:vconf}
\end{align}
where $\pi_{\mathbb T}$ denotes the canonical projection from ${\mathbb Z}^d$ onto $\mathbb T$. With (2.16) of \cite{int}, the law ${\mathbb Q}_u$ on $\{0,1\}^{{\mathbb Z}^d}$ of the indicator function of the vacant set at level $u \geq 0$ is characterized by the property
\begin{align}
{\mathbb Q}_u [\omega(x) = 1, \textrm{ for all } x \in K ] = \exp \{-u {\textup{ cap}(K)} \}, \textrm{ for all finite sets $K \subseteq {\mathbb Z}^d$,}\label{eq:int}
\end{align}
where $\omega(x)$, $x \in {\mathbb Z}^d$, are the canonical coordinates on $\{0,1\}^{{\mathbb Z}^d}$, and $\textup{cap}(K)$ the capacity of $K$, see (\ref{def:cap}) below. In this note, we show that the joint distribution of the vacant configurations in $M \geq 1$ distinct neighborhoods of distant points $x_1, \ldots, x_M$ at time $uN^d$ tends to the distribution of $M$ vacant sets of independent random interlacements at level $u$. This result has a similar flavor to Theorem 0.1 in \cite{S08}, which was proved in the context of random walk on a discrete cylinder.

\begin{theorem} \label{thm:main}
\textup{($u>0$, $d \geq 3$)} \newline
Consider $M \geq 1$ and for each $N \geq 1$, $x_1, \ldots, x_M$ points in $\mathbb T$ such that
\begin{align}
&\lim_N \inf_{1 \leq i \neq j \leq M} |x_i - x_j|_\infty = \infty.  \textrm{ Then} \label{thm1} \\ 
&\textrm{$(\omega_{x_1,uN^d}, \ldots , \omega_{x_M, uN^d})$ converges in distribution} \textrm{ to ${\mathbb Q}_u^{\otimes M}$ under $P$, as $N \to \infty$.} \label{thm2}
\end{align}
\end{theorem}

We now make some comments on the proof of Theorem~\ref{thm:main}. Standard arguments show that it suffices to show convergence of probabilities of the form $P\left[ H_B > uN^d \right]$ with $B=\bigcup_{i=1}^M (x_i + K_i)$  and finite subsets $K_i$ of ${\mathbb Z}^d$, where $H_B$ denotes the time until the first visit to the set $B \subseteq {\mathbb T}$ by the random walk. Since the size of the set $B$ does not depend on $N$, it is only rarely visited by the random walk for large $N$. It is therefore natural to expect that $H_B$ should be approximately exponentially distributed, see Aldous \cite{A89}, B2, p.~24. This idea is formalized by Theorem~\ref{thm:exp} below, quoted from Aldous and Brown \cite{AB}. Assuming that the distribution of $H_B$ is well approximated by the exponential distribution with expectation $E[H_B]$, the probability $P[H_B > uN^d]$ is approximately equal to $\exp \{ - uN^d/E[H_B] \}$. In order to show that this expression tends to the desired limit, which by (\ref{eq:int}) and (\ref{thm2}) is given by $\prod_{i=1}^M \exp \{ -u \textup{ cap}(K_i) \}$, one has to show that $N^d/E[H_B]$ tends to $\sum_{i=1}^M \textup{ cap}(K_i)$.

\vspace{12pt}
This task is accomplished with the help of the variational characterizations of the capacity of finite subsets of ${\mathbb Z}^d$ given by the Dirichlet and Thomson principles, see~(\ref{eq:con}) and (\ref{eq:res}). These principles characterize the capacity of a finite subset $A$ of ${\mathbb Z}^d$ as the infimum over all Dirichlet forms of functions of finite support on ${\mathbb Z}^d$ taking the value $1$ on $A$, resp. as the supremum over the reciprocal of energies dissipated by unit flows from $A$ to infinity. Aldous and Fill \cite{AF} show that very similar variational characterizations involving functions and flows on $\mathbb T$ hold for the quantity $N^d / E[H_A]$, see (\ref{eq:con'}) and (\ref{eq:res'}) below. In these two variational characterizations one optimizes the same quantities as in the Dirichlet and Thomson principles, over functions on the torus of zero mean, respectively over unit flows on the torus from $A$ to the uniform distribution. In the proof, we compare these two variational problems with the corresponding Dirichlet and Thomson principles and thus show the coincidence of $\lim_N N^d/E[H_B]$ with $\sum_{i=1}^M \textup{cap}(K_i)$. To achieve this goal, we construct a nearly optimal test function and a nearly optimal test flow for the variational problems on $\mathbb T$ using nearly optimal functions and a nearly optimal flow for the corresponding Dirichlet and Thomson principles.

In the case of the Dirichlet principle, this construction is rather simple and only involves shifting the functions on ${\mathbb Z}^d$ whose Dirichlet forms are almost $\textup{cap}(K_i)$ to the points $x_i$ on the torus, adding and rescaling them. In the Thomson principle, we identify the torus with a box in ${\mathbb Z}^d$ and consider the unit flow from $B$ to infinity on ${\mathbb Z}^d$ with dissipated energy equal to $\textup{cap}(B)^{-1}$. To obtain a flow on ${\mathbb T}$, we first restrict the flow to the box. The resulting flow then leaves charges at the boundary. In order to obtain a nearly optimal flow from $B$ to the uniform distribution for the variational problem (\ref{eq:res'}) on the torus, these charges need to be redirected such that they become uniformly distributed on $\mathbb T$, with the help of an additional flow of small energy.

\vspace{12pt}
The article is organized as follows:

In section \ref{sec:pre}, we state the preliminary result on the approximate exponentiality of the distribution of $H_B$ and introduce the variational characterizations required.

In section \ref{sec:proof}, we prove Theorem~\ref{thm:main}.

\vspace{12pt}
Finally, we use the following convention concerning constants: Throughout the text, $c$ or $c'$ denote positive constants which only depend on the dimension $d$, with values changing from place to place. Dependence of constants on additional parameters appears in the notation. For example, $c(M)$ denotes a constant depending only on $d$ and $M$. 

\paragraph{Acknowledgments.} The author is grateful to Alain-Sol Sznitman for proposing the problem and for helpful advice.

\section{Preliminaries} \label{sec:pre}

In this section, we introduce some notation and results required for the proof of Theorem~\ref{thm:main}. We denote the $l_1$ and $l_\infty$-distances on $\mathbb T$ or ${\mathbb Z}^d$ by $|.|_1$ and $|.|_\infty$. For any points $x, x'$ in $\mathbb T$ or ${\mathbb Z}^d$, we write $x \sim x'$ if $x$ and $x'$ are neighbors with respect to the natural graph structure, i.e. if $|x-x'|_1 = 1$. For subsets $A$ and $B$ of ${\mathbb T}$ or ${\mathbb Z}^d$, we write $d(A,B)$ for their mutual distance induced by $|.|_\infty$, i.e.~$d(A,B) = \inf \{|x-x'|_\infty: x \in A, x' \in B\}$, $\textup{int} A = \{x \in A: x' \in A \textrm{ for all } x' \sim x \}$, as well as $\partial_{int} A$ for the interior boundary, i.e.~$\partial_{int} A = A \setminus \textup{int} A$, and $|A|$ for the number of points in $A$.

We obtain a continuous-time random walk $(X_{\eta_t})_{t \geq 0}$ by defining the Poisson process $(\eta_t)_{t \geq 0}$ of parameter 1, independent of $X$. We write $P^{{\mathbb Z}^d}$ for the law of the simple random walk on ${\mathbb Z}^d$ and also denote the corresponding canonical process on ${\mathbb Z}^d$ as $X_.$, which should not cause any confusion. For $t \geq 0$, the set of points visited by the random walk until time $[t]$ is denoted by $X_{[0,t]}$, i.e.~$X_{[0,t]} = \{X_0, X_1, \ldots, X_{[t]} \}$. For any subset $A$ of  ${\mathbb T}$ or of ${\mathbb Z}^d$, we define the discrete- and continuous-time entrance times $H_A$ and ${\bar H}_A$ as
\begin{align}
H_A = \inf \{n \geq 0: X_n \in A \} \quad \textrm{and} \quad {\bar H}_A = \inf \{t \geq 0: X_{\eta_t} \in A \},  \label{def:ent}
\end{align}
as well as the hitting time
\begin{align}
{\tilde H}_A = \inf \{n \geq 1: X_n \in A \}.  \label{def:hit}
\end{align}
Note that by independence of $X$ and $\eta$, one then has
\begin{align}
E[{\bar H}_A] &= \sum_{n=0}^\infty P[H_A=n] E[\inf \{ t \geq 0: \eta_t = n \} ] = \sum_{n=0}^\infty P[H_A=n]n= E[H_A]. \label{eq:expent} 
\end{align}
The Green function of the simple random walk on ${\mathbb Z}^d$ is defined as
\begin{align}
g(x,x') = E_x^{{\mathbb Z}^d} \biggl[ \sum_{n=0}^\infty 1\{X_n = x'\} \biggr], \quad \textrm{for } x, x' \in {\mathbb Z}^d. \label{def:g}
\end{align}
In order to motivate the remaining definitions given in this section, we quote a result from Aldous and Brown \cite{AB}, which estimates the difference between the distribution of ${\bar H}_A$ and the exponential distribution. The following theorem appears as Theorem 1 in \cite{AB} for general irreducible, finite-state reversible continuous-time Markov chains and is stated here for the continuous-time random walk $(X_{\eta_t})_{t \geq 0}$ on $\mathbb T$, cf.~the remark after the statement.

\begin{theorem} \label{thm:exp}
\textup{($d \geq 1$)}
\newline
For any subset $A$ of $\mathbb T$ and $t \geq 0$,
\begin{align}
\left| P[{\bar H}_A > t E[H_A]] - \exp \{-t\} \right| \leq cN^2/E[H_A]. \label{eq:exp}
\end{align}
\end{theorem}

\begin{remark}
\textup{
In (\ref{eq:exp}), we have used (\ref{eq:expent}) to replace $E[{\bar H}_A]$ by $E[H_A]$, as well as the fact that the spectral gap of the transition matrix of the random walk $X$ on $\mathbb T$ is bounded from below by $cN^{-2}$. One of the many ways to show this last claim is to first find (by an explicit calculation of the eigenvalues, see, for example, \cite{AF}, Chapter 5, Example 7) that in dimension $d=1$, the spectral gap is given by $\rho_1= 1- \cos(2\pi/N) \geq cN^{-2}$. The $d$-dimensional random walk $X$ on $\mathbb T$ can be viewed as a $d$-fold product chain, from which it follows that its spectral gap is equal to $\rho_1 / d \geq cN^{-2}$, cf.~\cite{SC}, Lemma 2.2.11.}
\end{remark}

The main aim in the proof of Theorem \ref{thm:main} will be to obtain the limit as $N$ tends to infinity of probabilities of the form $P[{\bar H}_A> uN^d]$. In view of (\ref{eq:exp}), it is thus helpful to understand the asymptotic behavior of expected entrance times. To this end, we will use variational characterizations of expected entrance times involving Dirichlet forms and flows, which we now define.
For a real-valued function $f$ on $E={\mathbb T}$ or ${\mathbb Z}^d$, we define the Dirichlet form ${\mathcal E}_E$ as
\begin{align}
{\mathcal E}_E (f,f) = \frac{1}{2} \sum_{x \in E} \sum_{x' \sim x} \left( f(x) - f(x') \right)^2 \frac{1}{2d}. \label{def:dir}
\end{align}
We write $C_c$ for the set of real-valued functions on ${\mathbb Z}^d$ of finite support and denote the supremum norm of any function $f$ by $|f|_\infty$. The integral of a function $f$ on $\mathbb T$ with respect to the uniform distribution $\nu$ is denoted by $\nu(f)$ (i.e.~$\nu(f) = N^{-d} \sum_{x \in {\mathbb T}} f(x)$).
A flow $I = (I_{x,x'})$ on the edges of $E={\mathbb T}$ or ${\mathbb Z}^d$ is a real-valued function on $E^2$ satisfying
\begin{align}
I_{x,x'} = \left\{ \begin{array}{cl} -I_{x',x} & \textrm{if } x \sim x', \\
0 & \textrm{otherwise.} \end{array} \right. \label{def:flow}
\end{align}
Given a flow $I$, we write $|I|_\infty = \sup_{x,x' \in E} |I_{x,x'}|$ and define its dissipated energy as
\begin{align}
(I,I)_E = \frac{1}{2} \sum_{x \in E} \sum_{x' \in E} I_{x,x'}^2 2d. \label{def:energy}
\end{align}
The set of all flows on the edges of $E$ with finite energy is denoted by $F(E)$. For a flow $I \in F(E)$, the divergence $\textup{div} I$ on $E$ associates to every point in $E$ the net flow out of it,
\begin{align}
\textup{div} I (x)= \sum_{x' \sim x} I_{x,x'}, \quad x \in E. \label{def:div} 
\end{align}
The net flow out of a finite subset $A \subseteq E$ is denoted by
\begin{align}
I(A) = \sum_{x \in A} \sum_{x' \sim x} I_{x,x'} = \sum_{x \in A} \textup{div} I (x).
\end{align}
From Aldous and Fill, Chapter 3, Proposition 41, it is known that $N^d/E[H_A]$ is given by the infimum over all Dirichlet forms of functions on $\mathbb T$ of zero mean and equal to $1$ on $A$, and by the supremum over the reciprocals of energies dissipated by unit flows from $A$ to the uniform distribution $\nu$:
\begin{align}
N^d / E[H_A] &= \inf \bigl\{ {\mathcal E}_{\mathbb T}(f,f): f=1 \textrm{ on } A, \nu (f) =0 \bigr\} \label{eq:con'} \\
& = \sup \bigl\{ 1/(I,I)_{\mathbb T}: I \in F({\mathbb T}), I(A)= 1- |A| N^{-d}, \label{eq:res'} \\
& \qquad \qquad \textup{div} I(x) = - N^{-d} \textrm{ for all } x \in {\mathbb T} \setminus A \bigr\}. \nonumber
\end{align}
These variational characterizations are very similar to the Dirichlet and Thomson principles characterizing the capacity of finite subsets of ${\mathbb Z}^d$, to which we devote the remainder of this section. A set $A \subseteq {\mathbb Z}^d$ has its associated equilibrium measure $e_A$ on ${\mathbb Z}^d$, defined as 
\begin{align}
e_A(x) = \left\{ \begin{array}{ll} P^{{\mathbb Z}^d}_x[{\tilde H}_A = \infty] & \textrm{if } x \in A, \\
0 & \textrm{if } x \in {\mathbb Z}^d \setminus A. \end{array} \right.
\end{align}
The capacity of $A$ is defined as the total mass of $e_A$,
\begin{align}
\textup{cap}(A) = e_A ({\mathbb Z}^d). \label{def:cap}
\end{align}
For later use, we record that the following expression for the hitting probability of $A$ is obtained by conditioning on the time and location of the last visit to $A$ and applying the simple Markov property:
\begin{align}
P^{{\mathbb Z}^d}_x[H_A < \infty] = \sum_{x' \in A} g(x,x') e_A(x'), \quad \textrm{for } x \in {\mathbb Z}^d. \label{eq:heq}
\end{align}
The Dirichlet and Thomson principles assert that $\textup{cap}(A)$ is obtained by minimizing the Dirichlet form over all functions on ${\mathbb Z}^d$ of compact support equal to $1$ on $A$, or by maximizing the reciprocal of the energy dissipated by so-called unit flows from $A$ to infinity: 

\begin{proposition} \label{thm:di}
\textup{($d \geq 3$, $A \subseteq {\mathbb Z}^d$, $|A| < \infty$)}
\begin{align}
\textup{ cap}(A) &= \inf \left\{{\mathcal E}_{{\mathbb Z}^d}(f,f): f \in C_c, \textrm{ } f = 1 \textrm{ on } A \right\} \label{eq:con} \\
&= \sup \{1/(I,I)_{{\mathbb Z}^d}: I \in F({\mathbb Z}^d), I(A)=1, \textup{div} I(x) = 0, \textrm{ for all } x \in {\mathbb Z}^d \setminus A \}. \label{eq:res} 
\end{align}
Moreover, the unique maximizing flow $I^A$ in the variational problem (\ref{eq:res}) satisfies
\begin{align}
I^A_{x,x'} = - (2d \textup{ cap}(A))^{-1} ( P^{{\mathbb Z}^d}_{x'}[H_A < \infty] - P^{{\mathbb Z}^d}_x [H_A < \infty] ), \quad x \sim x' \in {\mathbb Z}^d. \label{eq:maxflow} 
\end{align}
\end{proposition}

\begin{proof}
By collapsing the set $A$ to a point (see for example \cite{AF}, Chapter 2, Section 7.3), it suffices to consider a general transient graph $G$ instead of ${\mathbb Z}^d$ and $A = \{a\}$, for a vertex $a$ in $G$. The proof for this case can be found in \cite{S94}: Theorem 3.41 shows (\ref{eq:con}) above and Theorem 3.25 with $\iota = 1_{\{a\}}$ (in the notation of \cite{S94}; allowed by Theorem 3.30 and transience of the simple random walk in dimension $d \geq 3$), combined with Corollary 2.14, yields the above claims (\ref{eq:res}) and (\ref{eq:maxflow}).
\end{proof}

\section{Proof} \label{sec:proof}

With the results of the last section, we are now ready to give the proof of Theorem~\ref{thm:main}. 

\begin{proof}[Proof of Theorem \ref{thm:main}.]
Take any finite subsets $K_1, \ldots K_M$ of ${\mathbb Z}^d$ and, using the notations of the theorem, set $B = \bigcup_{i=1}^M (x_i + K_i)$. Note that the collection of events $\{\omega(x) = 1 \textrm{ for all } x \in K\}$ as $K$ varies over finite subsets of ${\mathbb Z}^d$ forms a $\pi$-system generating the canonical product $\sigma$-algebra on $\{0,1\}^{{\mathbb Z}^{d}}$. By compactness of the set of probability measures on $(\{0,1\}^{{\mathbb Z}^{d}})^M$, our claim will follow once we show that
\begin{align}
\lim_N P[H_B>uN^d] = \prod_{i=1}^M e^{-u \textup{ cap}(K_i)}. \label{eq:thm1} 
\end{align}
As we now explain, we can replace $H_B$ by its continuous-time analog ${\bar H}_B$ in (\ref{eq:thm1}). Indeed, assume (\ref{eq:thm1}) holds with $H_B$ replaced by ${\bar H}_B$. By the law of large numbers, one has $\eta_t/t \to 1$ $a.s.$, as $t$ tends to infinity (see, for example \cite{durrett}, Chapter 1, Theorem 7.3), and it then follows that, for $0<\epsilon <u$,
\begin{align*}
\limsup_{N} P[H_B>uN^d] &= \limsup_N P[X_{[0,uN^d]} \cap B = \emptyset] \\
&\leq \limsup_N P[X_{[0,\eta_{(u-\epsilon)N^d}]} \cap B = \emptyset] \\
&= \limsup_N P[{\bar H}_B > (u-\epsilon) N^d] = \prod_{i=1}^M e^{-(u-\epsilon) \textup{ cap}(K_i)},
\end{align*}
and similarly,
\begin{align*}
\liminf_N P[H_B>uN^d] &\geq \liminf_N P[X_{[0,\eta_{(u+\epsilon)N^d}]} \cap B = \emptyset] \\
&= \liminf_N P[{\bar H}_B > (u+\epsilon)N^d] = \prod_{i=1}^M e^{-(u+\epsilon) \textup{ cap}(K_i)}.
\end{align*}
Letting $\epsilon$ tend to $0$ in the last two estimates, one deduces the desired result. By the above observations and (\ref{eq:exp}) with $A=B$ and $t=uN^d/E[H_B]$, all that is left to prove is that
\begin{align}
\lim_N \frac{N^d}{E[H_B]} = \sum_{i=1}^M \textup{cap}(K_i). \label{eq:thm3}
\end{align}
The claim (\ref{eq:thm3}) will be shown by using the variational characterizations (\ref{eq:con'}), (\ref{eq:res'}), (\ref{eq:con}) and (\ref{eq:res}). To this end, we map the torus $\mathbb T$ to a subset of ${\mathbb Z}^d$ in the following way:  We choose a point $x_*$ in $\mathbb T$ as the origin and then define the bijection $\psi: {\mathbb T} \to {\mathbb T}' = \{0, \ldots, N-1\}^d$ such that $\pi_{\mathbb T}(\psi(x_*+x)) = x$ for $x \in {\mathbb T}$, where $\pi_{\mathbb T}$ denotes the canonical projection from ${\mathbb Z}^d$ onto $\mathbb T$. Since there are only $M$ points $x_i$, we can choose $x_*$ in such a way that in ${\mathbb T}' \subseteq {\mathbb Z}^d$, $\psi(B)$ remains at a distance of order $N$ from the interior boundary of ${\mathbb T}'$, i.e.~such that for $N \geq c(M)$, 
\begin{align}
d(\psi(B), \partial_{int} {{\mathbb T}'}) \geq c'(M) N. \label{eq:thm5.1}
\end{align}
We define the subsets $C$ and $S$ of $\mathbb T$ as the preimages of $\textup{int} {\mathbb T}'$ and $\partial_{int} {\mathbb T}'$ under $\psi$, i.e.
\begin{align}
C= \psi^{-1}( \textup{int} {\mathbb T}'), \quad \textrm{and} \quad S = \psi^{-1}(\partial_{int} {\mathbb T}'). \label{def:cs}
\end{align}
For $\epsilon >0$, we now consider functions $f_i \in C_c$ (see above (\ref{def:flow})) such that $f_i = 1$ on $K_i$ and
\begin{align}
{\mathcal E}_{{\mathbb Z}^d} (f_i,f_i) \leq \textup{ cap}(K_i) + \epsilon, \textrm{ for } i=1, \ldots, M, \textrm{ cf.~(\ref{eq:con}).} \label{eq:thm4}
\end{align}
Defining $\tau_x: {\mathbb T} \to {\mathbb Z}^d$ by $\tau_x(x')= \psi(x')-\psi(x)$ for $x,x' \in {\mathbb T}$, we construct the function $f$ by shifting the functions $f_i$ to the points $x_i$, subtracting their means and rescaling so that $f$ equals $1$ on $B$ (for large $N$):
\begin{align*}
f = \frac{\sum_{i=1}^M f_i \circ \tau_{x_i} - \nu \left( \sum_{i=1}^M f_i \circ \tau_{x_i} \right)} {1-\nu \left( \sum_{i=1}^M f_i \circ \tau_{x_i} \right)}.
\end{align*}
Note that by the hypothesis (\ref{thm1}) and our choice (\ref{eq:thm5.1}) of the origin, the finite supports of the functions $f_i(. - \psi(x_i))$ intersect neither each other nor $\partial_{int} {\mathbb T}'$ for $N \geq c(M)$. One can then easily check that for $N \geq c(M, \epsilon)$ we have $f=1$ on $B$ and $\nu(f) = 0$. It therefore follows from (\ref{eq:con'}) that
\begin{align*}
\limsup_N N^d/E[H_B] &\leq \limsup_N {\mathcal E}_{\mathbb T} (f,f) \\
&\stackrel{(f_i \in C_c, (\ref{thm1}))}{=} 
\sum_{i=1}^M {\mathcal E}_{{\mathbb Z}^d} (f_i,f_i) \stackrel{(\ref{eq:thm4})}{\leq} \sum_{i=1}^M \textup{cap}(K_i) + M\epsilon. 
\end{align*}
Letting $\epsilon$ tend to $0$, one deduces that
\begin{align}
\limsup_N N^d/E[H_B] \leq \sum_{i=1}^M \textup{cap}(K_i). \label{eq:thm5}
\end{align}
In order to show the other half of (\ref{eq:thm3}), we proceed similarly, with the help of the variational characterizations (\ref{eq:res'}) and (\ref{eq:res}). We consider the flow $I^{\psi(B)} \in F({\mathbb Z}^d)$ such that 
\begin{align}
&I^{\psi(B)}( \psi(B))=1, \label{eq:thm5.2} \\
& \textup{div} I^{\psi(B)}(z) = 0 \textrm{ for all $z \in {\mathbb Z}^d \setminus \psi(B)$, and } \label{eq:thm5.3} \\
&1/(I^{\psi(B)},I^{\psi(B)})_{{\mathbb Z}^d} = \textup{ cap}(\psi(B)), \textrm{ cf.~(\ref{eq:res}), (\ref{eq:maxflow}).} \label{eq:thm6}
\end{align}
The aim is to now construct a flow of similar total energy satisfying the conditions imposed in (\ref{eq:res'}). To this end, we first define the flow $I^{*} \in F({\mathbb T})$ by restricting the flow $I^{\psi(B)}$ to ${\mathbb T}'$, i.e.~we set
\begin{align}
I^{*}_{x,x'} = I^{\psi(B)}_{\psi(x),\psi(x')} \textrm{ for } x,x' \in {\mathbb T}. \label{eq:thm6.0}
\end{align}
We now need a flow $J \in F({\mathbb T})$ such that $I^{*} + J$ is a unit flow from $A$ to the uniform distribution on $\mathbb T$. Essentially, $J$ has to redirect some of the charges $(\textup{div} I^*) 1_{S}$ left by $I^{*}$ on the set $S$, such that these charges become uniformly distributed on the torus, and the energy dissipated by $J$ has to decay as $N$ tends to infinity. The following proposition yields the required flow $J$:
\begin{proposition} \label{pr:j}
\textup{($d \geq 1$)}
\newline
There is a flow $J \in F({\mathbb T})$ such that
\begin{align}
& \textup{div} J(x) + (\textup{div} I^{*}) 1_{S}(x) = - N^{-d}, \textrm{ for any $x \in {\mathbb T}$, and} \label{eq:j1} \\
& |J|_\infty \leq c(M) N^{1-d}. \label{eq:j2}
\end{align}
\end{proposition}
Before we prove Proposition~\ref{pr:j}, we show how it enables to complete the proof of Theorem~\ref{thm:main}. Let us check that for large $N$, the flow $I^{*} + J$ satisfies the hypotheses imposed in (\ref{eq:res'}) with $A=B$. Since by (\ref{eq:thm5.1}), $\psi (B)$ is contained in $\textup{int} {\mathbb T}'$ for $N \geq c(M)$, one has for such $N$,
\begin{align*}
(I^{*}+J)(B) \stackrel{(\ref{eq:thm6.0})}{=} I^{\psi(B)}(\psi(B)) + J(B) \stackrel{(\ref{eq:thm5.2}), (\ref{eq:j1})}{=} 1 - |B|N^{-d}.
\end{align*}
Moreover, for any $N \geq c(M)$ and $x \in {\mathbb T} \setminus B$, 
\begin{align*}
\textup{div} (I^{*} + J)(x) &\stackrel{(\ref{eq:thm6.0})}{=} (\textup{div} I^{\psi(B)}) 1_{\textup{int} {\mathbb T}'} (\psi(x)) + (\textup{div} I^{*})1_{S}(x) + \textup{div} J(x) \\
&\stackrel{(\ref{eq:thm5.3}), (\ref{eq:j1})}{=} - N^{-d}.
\end{align*}
The flow $I^{*} + J$ is hence included in the collection on the right-hand side of (\ref{eq:res'}) with $A=B$ and it follows with the Minkowski inequality that
\begin{align}
E[H_B]N^{-d} &\leq (I^{*}+J,I^{*}+J)_{\mathbb T} \leq \left( (I^{*},I^{*})_{\mathbb T}^{\frac{1}{2}} + (J,J)_{\mathbb T}^{\frac{1}{2}} \right)^2. \label{eq:thm10}
\end{align}
By the bound (\ref{eq:j2}) on $|J|_\infty$, one has $(J,J)_{\mathbb T} \leq c(M) (N^{1-d})^2 N^d = c(M) N^{2-d}$. 
Inserting this estimate together with $$(I^{*},I^{*})_{\mathbb T} \stackrel{(\ref{eq:thm6.0})}{\leq} (I^{\psi(B)},I^{\psi(B)})_{{\mathbb Z}^d} \stackrel{(\ref{eq:thm6})}{=} 1/ \textup{cap}(\psi(B))$$ into (\ref{eq:thm10}), we deduce that
\begin{align}
E[H_B] N^{-d} &\leq \left( \textup{cap}(\psi(B))^{-\frac{1}{2}} + c(M) N^{-(d-2)/2} \right)^2. \label{eq:thm11}
\end{align}
Finally, we claim that
\begin{align}
\lim_N \textup{cap}(\psi(B)) = \sum_{i=1}^M \textup{cap}(K_i). \label{eq:thm12}
\end{align}
Indeed, the standard Green function estimate from \cite{lawler}, p.~31, (1.35) implies that for $d \geq 3$, $$P^{{\mathbb Z}^d}_x [H_{x'} < \infty] \leq g(x,x') \leq c|x-x'|_\infty^{2-d}, \quad x, x' \in {\mathbb Z}^d,$$ and claim (\ref{eq:thm12}) follows by assumption (\ref{thm1}) and the definition (\ref{def:cap}) of the capacity. Combining (\ref{eq:thm11}) with (\ref{eq:thm12}), one infers that for $d \geq 3$,
\begin{align*}
\limsup_N E[H_B] N^{-d} &\leq \biggl(\sum_{i=1}^M \textup{cap}(K_i) \biggr)^{-1}.
\end{align*}
Together with (\ref{eq:thm5}), this shows (\ref{eq:thm3}) and therefore completes the proof of Theorem~\ref{thm:main}. 
\end{proof}

It only remains to prove Proposition~\ref{pr:j}.
\begin{proof}[Proof of Proposition~\ref{pr:j}.]
The task is to construct a flow $J$ distributing the charges \\ $(\textup{div}I^*)1_{S}$ uniformly on $\mathbb T$, observing that we want the estimate (\ref{eq:j2}) to hold. To this end, we begin with an estimate on the order of magnitude of $\textup{div} I^*(x)$, for $x \in S$ and $N \geq c(M)$, where we sum over all neighbors $z$ of $\psi(x)$ in ${\mathbb Z}^d \setminus {\mathbb T}'$:
\begin{align}
&\bigl|\textup{div} I^{*}(x) \bigr| \stackrel{(\ref{eq:thm6.0})}{=} \biggl|\textup{div} I^{\psi(B)}(\psi(x)) - \sum_{z} I^{\psi(B)}_{\psi(x),z} \biggr| \stackrel{(\ref{eq:thm5.1}), (\ref{eq:thm5.3})}{\leq} \sum_{z} \bigl| I^{\psi(B)}_{\psi(x),z} \bigr| \label{eq:j5}\\
&\quad \stackrel{(\ref{eq:maxflow})}{\leq} c \sum_{z} \textup{ cap}(\psi(B))^{-1} \left| P_{z}^{{\mathbb Z}^d} [H_{\psi(B)}<\infty] - P^{{\mathbb Z}^d}_{\psi(x)}[H_{\psi(B)}<\infty] \right| \nonumber\\  
&\quad \stackrel{(\ref{eq:heq})}{\leq} c(M) N^{1-d}, \quad \textrm{for } x \in S, \nonumber
\end{align}
where we have also used the estimate on the Green function of \cite{lawler}, Theorem~1.5.4, together with (\ref{eq:thm5.1}), for the last line.
The required redirecting flow $J$ will be constructed as the sum of two flows, $K$ and $L$, both of which satisfy the estimate (\ref{eq:j2}). The purpose of $K$ is to redirect the charges $(\textup{div}I^*)1_{S}$, in such a way that the magnitude of the resulting charge at any given point is then bounded by $c(M)N^{-d}$, hence decreased by a factor of $N^{-1}$, cf.~(\ref{eq:j5}). Then, the flow $L$ will be used to distribute the resulting charges uniformly on $\mathbb T$. The existence of the flow $L$ will be a consequence of the following lemma (recall our convention concerning constants described at the end of the introduction and that $\nu$ denotes the uniform distribution on $\mathbb T$, cf.~above (\ref{def:flow})):
\begin{lemma} \label{lem:f}
\textup{($d \geq 1$)}
\newline
For any function $h: {\mathbb T} \to {\mathbb R}$, there is a flow $L^h \in F({\mathbb T})$, such that 
\begin{align}
&(\textup{div} L^h + h)(x) = \nu(h), \textrm{ for any } x \in {\mathbb T}, \textrm{ and} \label{f3} \\
&|L^h|_\infty \leq cN|h|_\infty. \label{f4}
\end{align}
\end{lemma}
\begin{proof}[Proof of Lemma~\ref{lem:f}.]
We construct the flow $L^h$ by induction on the dimension $d$, and therefore write ${\mathbb T}_d$ rather than $\mathbb T$ throughout this proof. Furthermore, we denote the elements of ${\mathbb T}_d$ using the coordinates of ${\mathbb T}'_d$ as $\{ (i_1, \ldots, i_d): 0 \leq i_j \leq N-1 \}$.  

In order to treat the case $d=1$, define the flow $L^h$ by letting the charges defined by $h$ flow from $0$ to $N-1$, such that the same charge is left at any point. Precisely, we set $L^h_{N-1, 0} = 0$ and $L^h_{i,i+1} = \sum_{j=0}^i (\nu(h) - h(j))$ for $i = 0, \ldots, N-2$ (the values in the opposite directions being imposed by the condition (\ref{def:flow}) on a flow). The flow $L^h$ then has the required properties (\ref{f3}) and (\ref{f4}).

Assume now that $d \geq 2$ and that the statement of the lemma holds in any dimension $<d$. Applying the one-dimensional case on every fiber $\{(0,y), \ldots, (N-1,y)\} \cong {\mathbb T}_1$, $y \in {\mathbb T}_{d-1}$, with the function $h^1$ defined by $h^1(.,y) = h(.,y)$, one obtains the flows $L^y$ supported by the edges of $\{(0,y), \ldots, (N-1,y)\}$, such that for any $i \in {\mathbb T}_1$,
\begin{align}
&(\textup{div}L^y + h) (i,y) = N^{-1} \sum_{j=0}^{N-1} h(j,y) \textrm{ and} \label{eq:f7} \\
&|L^y|_\infty \leq cN|h|_\infty. \label{eq:f8}
\end{align}
We now apply the induction hypothesis on the slices ${\mathbb S}_i = \{(i,y): y \in {\mathbb T}_{d-1}\} \cong {\mathbb T}_{d-1}$, $i \in {\mathbb T}_1$, with the function $h^2$ given by $h^2(i,.) = N^{-1} \sum_{j=0}^{N-1} h(j,.)$. For any $0 \leq i \leq N-1$, we thus obtain a flow $L^i$ supported by the edges of ${\mathbb S}_i$, such that for any $y \in {\mathbb T}_{d-1}$,
\begin{align}
&\textup{div}L^i(i,y) + N^{-1} \sum_{j=0}^{N-1} h(j,y) = N^{-(d-1)} \sum_{y' \in {\mathbb T}_{d-1}} h^2(i,y') = \nu(h) \textrm{ and} \label{eq:f11} \\
&|L^i|_\infty \leq cN|h|_\infty. \label{eq:f12}
\end{align}
Then equations (\ref{eq:f7})-(\ref{eq:f12}) imply that the flow
$L^h = \sum_{i=0}^{N-1} L^i + \sum_{y \in {\mathbb T}_{d-1}} L^y $
has the required properties. Indeed, the flows $L^y$ have disjoint supports, as do the flows $L^i$, and therefore the estimate (\ref{f4}) on $|L^h|_\infty$ follows from (\ref{eq:f8}) and (\ref{eq:f12}). Finally, for any $x=(i,y) \in {\mathbb T}_1 \times {\mathbb T}_{d-1} = {\mathbb T}_d$, (\ref{eq:f7}) and (\ref{eq:f11}) together yield
\begin{align*}
(\textup{div}L^h + h)(x) = \textup{div}L^i(i,y) + \textup{div}L^y (i,y) + h(i,y) = \nu(h),
\end{align*}
hence (\ref{f3}). This concludes the proof of Lemma~\ref{lem:f}.
\end{proof}

\noindent We now complete the proof of Proposition~\ref{pr:j}. To this end, we construct the auxiliary flow $K$ described above Lemma~\ref{lem:f}. Set $g = (\textup{div} I^*) 1_S$. Writing $e_1, \ldots, e_d$ for the canonical basis of ${\mathbb R}^d$, choose a mapping $e: S \to \{\pm e_1, \ldots, \pm e_d\}$ such that ${\mathbb F}'_x \stackrel{(\textrm{def.})}{=} \{\psi(x), \psi(x) + e(x), \ldots, \psi(x)+(N-1)e(x) \} \subseteq {\mathbb T}'$ (whenever there are more than one possible choices for $e(x)$, take one among them arbitrarily), and define the fiber ${\mathbb F}_x = \psi^{-1} ({\mathbb F}'_x)$. 

Observe that any point $x \in \mathbb T$ only belongs to the $d$ different fibers $x + [0,N-1] e_i$, $i=1,\ldots,d$. Moreover, we claim that for any ${\mathbb F} \in \{{\mathbb F}_x\}_{x \in S}$, there are at most $2$ points $x \in S$ such that ${\mathbb F}_x = {\mathbb F}$. Indeed, suppose that ${\mathbb F}_x = {\mathbb F}_{x'}$ for $x, x' \in S$. Then $\psi({\mathbb F}_x) = \psi({\mathbb F}_{x'})$ implies that $\psi(x') = \psi(x) + ke(x)$ for some $k \in \{0, \ldots, N-1\}$ and that either $e(x) = e(x')$ or $e(x) = -e(x')$. If $e(x) = e(x')$, then for $\psi({\mathbb F}_{x'}) = \{ \psi(x) + ke(x), \psi(x) + (k+1)e(x), \ldots, \psi(x) + (k+N-1)e(x) \}$ to be a subset of ${\mathbb T}'$, we require $k=0$ (since $\psi(x) + Ne(x) \notin {\mathbb T}'$). Similarly, if $e(x) = -e(x')$ one needs $k=N-1$. Hence, $x'$ can only be equal to either $x$ or $x+(N-1)e(x)$. The above two observations on the fibers ${\mathbb F}_x$ together imply the crucial fact that any point in $\mathbb T$ belongs to a fiber ${\mathbb F}_x$ for at most $2d$ points $x \in S$.

We then define the flow $K^x$ from $x$ to $x+(N-1)e(x)$ distributing the charge $g(x)$ uniformly on the fiber ${\mathbb F}_{x}$. That is, the flow $K^x \in F({\mathbb T})$ is supported by the edges of ${\mathbb F}_{x}$, and characterized by $K^x_{x +(N-1)e(x), x} = 0$, $K^x_{x+i e(x), x+(i+1)e(x)} = -g(x) (N-(i+1))/N$ for $i = 0, \ldots, N-2$. Observe that then $|K^x|_\infty \leq |g|_\infty$ and $|\textup{div} K^x + g1_{\{x\}}|_\infty = |g(x)|/N \leq |g|_\infty/N$. Moreover, any point in $\mathbb T$ belongs to at most $2d$ fibers ${\mathbb F}_{x}$, hence to the support of at most $2d$ flows $K^x$. If we define the flow $K \in F({\mathbb T})$ as
$K = \sum_{x \in S} K^x,$
then we therefore have
\begin{align}
|K|_\infty \leq c \max_{x \in S} |K^x|_\infty \leq c |g|_\infty, \label{eq:f1}
\end{align}
as well as, for $x \in {\mathbb T}$,
\begin{align}
|(\textup{div} K + g)(x)| & \leq \sum_{x' \in S} |(\textup{div} K^{x'} + g1_{\{x'\}})(x)| \label{eq:f2} \\
& \leq |\textup{div} K^x + g1_{\{x\}}|_\infty 1_{S}(x) + \sum_{x' \neq x: x \in {\mathbb F}_{x'}} |\textup{div} K^{x'}(x)| \nonumber\\
& \leq c |g|_\infty/N. \nonumber
\end{align}
We claim that the flow $J = K+L^{\textup{div} K +g}$ has the required properties (\ref{eq:j1}) and (\ref{eq:j2}). Indeed, using the fact that $\nu (\textup{div} I) = 0$ for any flow $I \in F({\mathbb T})$,
\begin{align*}
\textup{div} J + g &= \textup{div}  L^{\textup{div} K +g} + \textup{div} K + g \\
&\stackrel{(\ref{f3})}{=} \nu (\textup{div} K + g) = \nu((\textup{div} I^*) 1_S) = - \nu( (\textup{div} I^*) 1_{C} ) \\
&\stackrel{(\ref{eq:thm6.0})}{=} -N^{-d} \sum_{z \in \textup{int} {\mathbb T}'} \textup{div} I^{\psi(B)}_z \stackrel{(\ref{eq:thm5.3})}{=} -N^{-d} I^{\psi(B)}(\psi(B)) \stackrel{(\ref{eq:thm5.2})}{=} -N^{-d}.
\end{align*}
Finally, the estimates (\ref{f4}), (\ref{eq:f1}) and (\ref{eq:f2}) imply that
\begin{align*}
|J|_\infty \leq |K|_\infty + |L^{\textup{div} K +g}|_\infty \leq c|g|_\infty \stackrel{(\ref{eq:j5})}{\leq} c(M) N^{1-d}.
\end{align*}
The proof of Proposition~\ref{pr:j} is thus complete.
\end{proof}

\end{document}